\documentclass[12pt]{article}
\usepackage{amsfonts}
\usepackage{amssymb}
\usepackage{eucal}

\newtheorem{cor}{Corollary}[section]
\newtheorem{theo}{Theorem}[section]

\newtheorem{prop}{Proposition}[section]
\newenvironment{proof}[1][Proof]{\noindent \textbf{#1.} }{\  \rule{0.5em}{0.5em}}

\input{tcilatex}

\begin{document}

\title{ On generalized Robertson-Walker spacetimes satisfying some curvature
condition }
\author{Kadri Arslan, Ryszard Deszcz, Ridvan Ezenta\d s, Marian Hotlo\'{s} \\
and Cengizhan Murathan}
\maketitle

\begin{abstract}
We give necessary and sufficient conditions for warped product manifolds
with $1$-dimensional base, and in particular, for generalized
Robertson-Walker spacetimes, to satisfy some generalized Einstein metric
condition. We also construct suitable examples of such manifolds. They are
quasi-Einstein or not.\footnotemark[1]
\end{abstract}

\begin{center}
\textsl{Dedicated to Professor Jacek Gancarzewicz}
\end{center}

\footnotetext[1]{%
The authors are supported by the grant F-2003/98 of the Uluda\v{g}
University of Bursa (Turkey). \newline
\textbf{Mathematics Subject Classification (2010).} Primary 53B20, 53B30,
53B50; Secondary 53C25, 53C50, 53C80. \newline
\textbf{Key words:} warped product, generalized Robertson-Walker spacetime,
quasi-Einstein manifold, generalized Einstein-metric condition,
pseudosymmetry type curvature condition, Ricci-pseudosymmetric hypersurface.}

\section{Introduction}

A semi-Rie\-man\-nian manifold $(M,g)$, $n = \dim M \geqslant 3$, is said to
be an \textsl{Einstein manifold} if at every point its Ricci tensor $S$ is
proportional to the metric tensor $g$, i.e. on $M$ we have 
\begin{equation}
S \ =\  \frac{\kappa }{n} \, g \,,  \label{einstein}
\end{equation}
where $\kappa $ is the scalar curvature of $(M,g)$. In particular, if $S$
vanishes identically on $M$ then $(M,g)$ is called a \textsl{Ricci flat
manifold}. If at every point of $M$ its Ricci tensor satisfies $\mbox{rank}%
\, S \leqslant 1$ then $(M,g)$ is called a \textsl{Ricci-simple manifold}
(see e.g. \cite{DH2, G6}.

Let $(M,g)$, $n \geqslant 3$, be a semi-Rie\-man\-nian manifold and let $%
U_{S}$ be the set of all points of $M$ at which $S \neq \frac{\kappa }{n} \,
g$. The manifold $(M,g)$, $n \geqslant 3$, is said to be \textsl{%
quasi-Einstein} (see e.g. \cite{AKH} and \cite{DGHSaw} and references
therein) if at every point of $U_{S} \subset M$ we have $\mbox{rank}\, (S -
\alpha \, g) = 1$, for certain $\alpha \in {\mathbb{R}}$.

For the curvature tensor $R$ and the Weyl conformal curvature tensor $C$ of $%
(M,g)$, $n \geqslant 4$, we can define on $M$ the $(0,6)$-tensors $R \cdot C$
and $C \cdot R$. For precise definition of the symbols used, we refer to
Sections 2 and 3 of this paper, as well as to \cite{BDGHKV, DGHSaw, DGHV,
DHS}. It is obvious that for any Ricci flat, as well as conformally flat,
semi-Riemannian manifold $(M,g)$, $n \geqslant 4$, we have $R \cdot C - C
\cdot R = 0$. For non-Ricci flat Einstein manifold the tensor $R \cdot C - C
\cdot R$ is non-zero. Namely, any Einstein manifold $(M,g)$, $n \geqslant 4$%
, satisfies (\cite{DHS}, Theorem 3.1) 
\begin{eqnarray}
R \cdot C - C \cdot R \ = \  \frac{\kappa }{(n-1) n} \, Q(g,R) \ = \  \frac{%
\kappa }{(n-1) n} \, Q(g,C)\,,  \label{genein1}
\end{eqnarray}
i.e. at every point of $M$ the difference tensor $R \cdot C - C \cdot R$ and
the Tachibana tensor $Q(g,R)$, or $Q(g,C)$, are linearly dependent. We also
mention that for any semi-Riemannian manifold $(M,g)$, $n \geqslant 4$, we
have some identity (see eq. (\ref{ge})) which express the tensor $R \cdot C
- C \cdot R$ by some $(0,6)$-tensors. In particular, by making use of that
identity we can express the difference tensor of some hypersurfaces in space
forms by a linear combination of the Tachibana tensors $Q(g,R)$ and $Q(S,R)$
(\cite{DGHV}, Theorem 3.2; see also our Theorem 6.1(ii) and Proposition 4.1).

We also can investigate semi-Riemannian manifolds $(M,g)$, $n \geqslant 4$,
for which the difference tensor $R \cdot C - C \cdot R$ is expressed by one
of the following Tachibana tensors: $Q(g,R)$, $Q(g,C)$, $Q(S,R)$ or $Q(S,C)$%
. In this way we obtain four curvature conditions. First results related to
those conditions are given in \cite{DHS}. We refer to \cite{DGHSaw} for a
survey on this subject. Since these conditions are satisfied on any
semi-Riemannian Einstein manifold, they can be named \textsl{generalized
Einstein metric conditions} (cf. \cite{Besse}, Chapter XVI). In particular, (%
\ref{genein1}) is also a condition of this kind and in Section 2 we present
results on manifolds satisfying 
\begin{eqnarray}
R \cdot C - C \cdot R \ = \ L_{1} \, Q(g,R) \,.  \label{geneintsu}
\end{eqnarray}
In this paper we restrict our investigations to non-Einstein and
non-conformally flat semi-Riemannian manifolds $(M,g)$, $n \geqslant 4$,
satisfying on $U=\{x\in M\,:\,Q(S,R)\neq 0\ {\mbox{at}}\ x\}$ the condition 
\begin{equation}
R\cdot C - C\cdot R\ =\ L\, Q(S,R)\,,  \label{a1}
\end{equation}
where $L$ is some function on this set. We recall that at all points of a
semi-Riemannian manifold $(M,g)$, $n \geqslant 3$, at which its Ricci tensor 
$S$ is non-zero and $Q(S,R)=0$ we have (\cite{DD}, Theorem 4.1) 
\begin{equation}
R\cdot R\ =\ 0 \,,  \label{semisymmetry}
\end{equation}
i.e. such manifold is \textsl{semisymmetric}. We also recall that if 
\begin{equation}
R\cdot S\ =\ 0  \label{Ricci-semisymmetry}
\end{equation}
holds on a semi-Riemannian manifold then it is called \textsl{%
Ricci-semisymmetric}. The condition (\ref{a1}), under some additional
assumptions, was considered in \cite{H}. We have

\begin{theo}[\protect \cite{H}, Theorem 3.1 and Proposition 3.1]
(i) Let $(M,g)$, $n\geqslant 4$, be a non-conformally flat and non-Einstein
Ricci-semisymmetric manifold satisfying (\ref{a1}). Then on the set
consisting of all points of $M$ at which $L$ is non-zero we have $L=\frac{1}{%
n-2}$ and 
\begin{equation}
R(\mathcal{S }X,Y,Z,W)\ =\  \frac{\kappa}{n-1}\,R(X,Y,Z,W)\,.  \label{H1}
\end{equation}
(ii) Let $(M,g)$, $n\geqslant 4$, be a semi-Riemannian manifold satisfying (%
\ref{H1}). Then $(M,g)$ is a Ricci-semisymmetric manifold fulfilling (\ref%
{a1}) with $L=\frac{1}{n-2}$.
\end{theo}

We consider warped products $\overline{M} \times _{F} \widetilde{N}$ with $1$%
-dimensional base manifold $(\overline{M}, \bar{g})$ satisfying (\ref{a1}).
Evidently, generalized Robertson-Walker spacetimes are warped products of
this type. We investigate separately two cases when the fibre $(\widetilde{N}%
, \tilde{g})$ is either non-Einstein either Einstein manifold. In the first
case we prove that the associated function $L$ satisfies $L = \frac{1}{n-2}$
and we show that the warping function $F$ is a polynomial of the 2nd degree.
Moreover, $(\widetilde{N}, \tilde{g})$ satisfies some curvature condition
(see eq. (\ref{SR2})). In the second case, i.e. when $\overline{M} \times
_{F} \widetilde{N}$ is a quasi-Einstein manifold, we show that two subcases
are possible. The first one leads to the same $L$ and $F$ as in the
non-Einstein case. The second subcase leads to $L = \frac{1}{n-1}$ and to
another forms of $F$ (see Theorem 5.1). We also give converse statements.
Basing on these results we give examples of warped products satisfying (\ref%
{a1}). Finally, we mention that recently hypersurfaces in space forms having
the tensor $R \cdot C - C \cdot R$ expressed by some Tachibana tensors has
been investigated in \cite{DGPSS}. \newline

\noindent \textbf{Acknowledgments.} The first, third and fifth named author
would express their thanks to the Institute of Mathematics and Computer
Science of the Wroc\l aw University of Technology as well as to the
Department of Mathematics of the Wroc\l aw University of Environmental and
Life Sciences for the hospitality during their stay in Wroc\l aw. The second
and fourth named author would express their thanks to the Department od
Mathematics of the Uluda\v{g} University in Bursa for the hospitality during
their stay in Bursa.

\section{Preliminaries}

Throughout this paper all manifolds are assumed to be connected paracompact
manifolds of class $C^{\infty }$. Let $(M,g)$ be an $n$-dimensional
semi-Riemannian manifold and let $\nabla$ be its Levi-Civita connection and $%
\mathfrak{X}(M)$ the Lie algebra of vector fields on $M$. We define on $M$
the endomorphisms $X \wedge _{A} Y$ and ${\mathcal{R}}(X,Y)$ of $\mathfrak{X}%
(M)$ by 
\begin{eqnarray*}
(X \wedge _{A} Y)Z &=& A(Y,Z)X - A(X,Z)Y\,, \\
{\mathcal{R}}(X,Y)Z &=& \nabla _X \nabla _Y Z - \nabla _Y \nabla _X Z -
\nabla _{[X,Y]}Z\,,
\end{eqnarray*}
respectively, where $A$ is a symmetric $(0,2)$-tensor on $M$ and $X, Y, Z
\in \mathfrak{X}(M)$. The Ricci tensor $S$, the Ricci operator ${\mathcal{S}}
$ and the scalar curvature $\kappa $ of $(M,g)$ are defined by $S(X,Y) = tr
\{ Z \rightarrow {\mathcal{R}}(Z,X)Y \}$, $g({\mathcal{S}}X,Y)\, =\, S(X,Y)$
and $\kappa = tr\, {\mathcal{S}}$, respectively. The endomorphism ${\mathcal{%
C}}(X,Y)$ is defined by 
\begin{eqnarray*}
{\mathcal{C}}(X,Y)Z &=& {\mathcal{R}}(X,Y)Z - \frac{1}{n-2}(X \wedge _{g} {%
\mathcal{S}}Y + {\mathcal{S}}X \wedge _{g} Y - \frac{\kappa}{n-1}X \wedge
_{g} Y)Z\,,
\end{eqnarray*}
assuming that $n\geqslant 3$. Now the $(0,4)$-tensor $G$, the
Riemann-Christoffel curvature tensor $R$ and the Weyl conformal curvature
tensor $C$ of $(M,g)$ are defined by 
\begin{eqnarray*}
G(X_1,X_2,X_3,X_4) &=& g((X_1 \wedge _{g} X_2)X_3,X_4)\,, \\
R(X_1,X_2,X_3,X_4) &=& g({\mathcal{R}}(X_1,X_2)X_3,X_4)\,, \\
C(X_1,X_2,X_3,X_4) &=& g({\mathcal{C}}(X_1,X_2)X_3,X_4)\,,
\end{eqnarray*}
respectively, where $X_1,X_2,\ldots \in \mathfrak{X}(M)$. Further we define
the following sets $U_{R}=\{x\in M\, :\, R \neq \frac{\kappa }{(n-1) n} G\ {%
\mbox{at}}\ x\}$, $U_{S}=\{x\in M\, :\, S \neq \frac{\kappa }{n} g \ {%
\mbox{at}}\ x\}$ and $U_{C}=\{x\in M\, :\, C \neq 0\ {\mbox{at}}\ x\}$. It
is easy to see that $U_{S} \cap U_{C} \subset U_{R}$.

Let ${\mathcal{B}}(X,Y)$ be a skew-symmetric endomorphism of $\mathfrak{X}(M)
$ and let $B$ be a $(0,4)$-tensor associated with ${\mathcal{B}}(X,Y)$ by 
\begin{eqnarray}
B(X_1,X_2,X_3,X_4) &=& g({\mathcal{B}}(X_1,X_2)X_3,X_4)\,.  \label{DS5}
\end{eqnarray}
The tensor $B$ is said to be a \textsl{generalized curvature tensor} if 
\begin{eqnarray*}
& & B(X_1,X_2,X_3,X_4) + B(X_2,X_3,X_1,X_4) + B(X_3,X_1,X_2,X_4) \ =\ 0\,, \\
& & B(X_1,X_2,X_3,X_4) \ =\ B(X_3,X_4,X_1,X_2) \,.
\end{eqnarray*}

Let ${\mathcal{B}}(X,Y)$ be a skew-symmetric endomorphism of $\mathfrak{X}(M)
$ and let $B$ be the tensor defined by (\ref{DS5}). We extend the
endomorphism ${\mathcal{B}}(X,Y)$ to derivation ${\mathcal{B}}(X,Y) \cdot \, 
$ of the algebra of tensor fields on $M$, assuming that it commutes with
contractions and $\ {\mathcal{B}}(X,Y) \cdot \, f = 0$, for any smooth
function $f$ on $M$. Now for a $(0,k)$-tensor field $T$, $k \geqslant 1$, we
can define the $(0,k+2)$-tensor $B \cdot T$ by 
\begin{eqnarray*}
& & (B \cdot T)(X_1,\ldots ,X_k;X,Y) \ =\ ({\mathcal{B}}(X,Y) \cdot
T)(X_1,\ldots ,X_k) \\
&=& - T({\mathcal{B}}(X,Y)X_1,X_2,\ldots ,X_k) - \cdots - T(X_1,\ldots
,X_{k-1},{\mathcal{B}}(X,Y)X_k)\,.
\end{eqnarray*}
In addition, if $A$ is a symmetric $(0,2)$-tensor then we define the $(0,k+2)
$-tensor $Q(A,T)$, named a \textsl{Tachibana tensor} (\cite{DGPSS}), by 
\begin{eqnarray*}
& & Q(A,T)(X_1, \ldots , X_k; X,Y) \ =\ (X \wedge _{A} Y \cdot T)(X_1,\ldots
,X_k) \\
&=&- T((X \wedge _A Y)X_1,X_2,\ldots ,X_k) - \cdots - T(X_1,\ldots
,X_{k-1},(X \wedge _A Y)X_k)\, .
\end{eqnarray*}
In this manner we obtain the $(0,6)$-tensors $B \cdot B$ and $Q(A,B)$.
Setting in the above formulas ${\mathcal{B}} = {\mathcal{R}}$ or ${\mathcal{B%
}} = {\mathcal{C}}$, $T=R$ or $T=C$ or $T=S$, $A=g$ or $A=S$, we get the
tensors $R\cdot R$, $R\cdot C$, $C\cdot R$, $R\cdot S$, $Q(g,R)$, $Q(S,R)$, $%
Q(g,C)$ and $Q(g,S)$.

Let $B_{hijk}$, $T_{hijk}$, and $A_{ij}$ be the local components of
generalized curvature tensors $B$ and $T$ and a symmetric $(0,2)$-tensor $A$
on $M$, respectively, where $h,i,j,k,l,m,p,q \in \{ 1,2, \ldots , n \}$. The
local components $(B \cdot T)_{hijklm}$ and $Q(A,T)_{hijklm}$ of the tensors 
$B \cdot T$ and $Q(A,T)$ are the following 
\begin{eqnarray*}
(B \cdot T)_{hijklm} &=& g^{pq}(T_{pijk}B_{ qhlm}+ T_{hpjk}B_{ qilm} +
T_{hipk}B_{ qjlm}+ T_{hijp}B_{ qklm})\,,
\end{eqnarray*}
\begin{eqnarray}
Q(A,T)_{hijklm}&=&A_{hl}T_{ mijk} + A_{il}T_{ hmjk}+ A_{jl}T_{ himk}+
A_{kl}T_{ hijm} \\
&-& A_{hm}T_{ lijk}- A_{im}T_{ hljk}- A_{jm}T_{ hilk} - A_{km}T_{ hijl} \,. 
\nonumber  \label{qat}
\end{eqnarray}
For a symmetric $(0,2)$-tensor $E$ and a $(0,k)$-tensor $T$, $k \geqslant 2$%
, we define their Kulkarni-Nomizu product $E \wedge T$ by (e.g. see \cite%
{DeGlo1}) 
\begin{eqnarray*}
& &(E \wedge T )(X_{1}, X_{2}, X_{3}, X_{4}; Y_{3}, \ldots , Y_{k}) \\
&=& E(X_{1},X_{4}) T(X_{2},X_{3}, Y_{3}, \ldots , Y_{k}) + E(X_{2},X_{3})
T(X_{1},X_{4}, Y_{3}, \ldots , Y_{k} ) \\
& & - E(X_{1},X_{3}) T(X_{2},X_{4}, Y_{3}, \ldots , Y_{k}) - E(X_{2},X_{4})
T(X_{1},X_{3}, Y_{3}, \ldots , Y_{k})\,.
\end{eqnarray*}
According to \cite{DGPSS}, the tensor $E \wedge T$ is called a \textsl{%
Kulkarni-Nomizu tensor}. Clearly, the tensors $R$, $C$, $G$ and $E \wedge F$%
, where $E$ and $F$ are symmetric $(0,2)$-tensors, are generalized curvature
tensors.

A semi-Riemannian manifold $(M,g)$, $n \geqslant 3$, is said to be \textsl{%
locally symmetric} if $\nabla R = 0$ holds on $M$. It is obvious that the
last condition leads immediately to the integrability condition (\ref%
{semisymmetry}). Manifolds satisfying (\ref{semisymmetry}) are called \textsl%
{semisymmetric }. A weaker condition than (\ref{semisymmetry}) there is 
\begin{eqnarray}
R \cdot R = L_{R} Q(g,R) \,,  \label{pseudosym}
\end{eqnarray}
which is considered on $U_{R} \subset M$, hence $L_{R}$ is a function
uniquely determined on this set. On $M \setminus U_{R}$ we have $R \cdot R =
Q(g,R) = 0$. We note that $Q(g,R) = 0$ at a point if and only if $R = \frac{%
\kappa }{(n - 1) n} G$ at this point. A semi-Riemannian manifold $(M,g)$, $n
\geqslant 3$, is said to be \textsl{pseudosymmetric} if $(\ref{pseudosym})$
holds on $U_{R} \subset M$ (\cite{BDGHKV, D 10, DHV, HV}).

A semi-Riemannian manifold $(M,g)$, $n \geqslant 3$, is said to be \textsl{%
Ricci-symmetric} if $\nabla R = 0$ holds on $M$. It is obvious that the last
condition leads immediately to the integrability condition (\ref%
{Ricci-semisymmetry}). Manifolds satisfying (\ref{Ricci-semisymmetry}) are
called \textsl{Ricci-semisymmetric}. A weaker condition than (\ref%
{Ricci-semisymmetry}) there is 
\begin{eqnarray}
R \cdot S = L_{S} Q(g,S) \,,  \label{Ricci-pseudosym}
\end{eqnarray}
which is considered on $U_{S} \subset M$, hence $L_{S}$ is a function
uniquely determined on this set. On $M \setminus U_{S}$ we have $R \cdot S =
Q(g,S) = 0$. We note that $Q(g,S) = 0$ at a point if and only if (\ref%
{einstein}) holds at this point (cf. \cite{DD}, Lemma 2.1 (i)). A
semi-Riemannian manifold $(M,g)$, $n \geqslant 3$, is said to be \textsl{%
Ricci-pseudosymmetric} if $(\ref{Ricci-pseudosym})$ holds on $U_{S} \subset M
$ (\cite{BDGHKV, D 10, DHV, JHSV}).

Every locally symmetric, resp. semisymmetric and pseudosymmetric, manifold
is Ricci-symmetric, resp. Ricci-semisymmetric, Ricci-pseudosymmetric. In all
cases, the converse statements are not true. We refer to \cite{BDGHKV}, \cite%
{DGHSaw}, \cite{DHV} and \cite{HaVer7} for a wider presentation of results
related to these classes of manifolds.

A geometric interpretation of (\ref{pseudosym}), resp. (\ref{Ricci-pseudosym}%
), is given in \cite{HV}, resp. in \cite{JHSV}. Semi-Riemannian manifolds
for which their curvature tensor $R$ is expressed by a linear combination of
the Kulkarni-Nomizu tensors $S \wedge S$, $g \wedge S$ and $G$ are called
Roter type manifolds, see e.g. \cite{DeSche1} and references therein.
Precisely, a semi-Riemannian manifold $(M,g)$, $n \geqslant 4$, is said to
be a \textsl{Roter type manifold} if 
\begin{eqnarray}
R &=& \frac{\phi}{2} \, S \wedge S + \mu \, g \wedge S + \eta \, G \,,
\label{R}
\end{eqnarray}
holds on the set $U_{1}$ of all points of $U_{S} \cap U_{C} \subset M$ at
which $\mbox{rank}\, (S - \alpha \, g) \geqslant 2$, for every $\alpha \in {%
\mathbb{R}}$. It is easy to prove that the functions $\phi $, $\mu $ and $%
\eta $ are uniquely determined on $U_{1}$. Using (\ref{R}) and suitable
definitions we can verify that on $U_{1}$ the condition (\ref{pseudosym}) is
satisfied (e.g. see \cite{DeSche1}, eqs. (7) and (8); \cite{DGHSaw}, Theorem
6.7) with $L_{R} = \phi ^{-1} ( (n-2) (\mu ^{2} - \phi \eta) - \mu )$, and
that the difference tensor $R \cdot C - C \cdot R$ is expressed on $U_{1}$
by a linear combination of the tensors $Q(S,R)$, $Q(g,R)$ and $Q(S,G)$ (\cite%
{DGHV}, eq. (47)), or equivalently, by a linear combination of the tensors $%
Q(g,R)$ and $Q(S,G)$ (\cite{DGHV}, eq. (48)).

Semi-Riemannian manifolds $(M,g)$, $n \geqslant 4$, satisfying (\ref%
{geneintsu}) on $U_{S} \cap U_{C} \subset M$ were investigated in \cite{DH2}%
. Among other results it was proved: (i) $R \cdot C = C \cdot R = 0$ and $%
\mbox{rank}\, S = 1$ hold on $U_{S} \cap U_{C}$, provided that $(M,g)$ is a
quasi-Einstein manifold and (ii) (\ref{R}), with some special coefficients $%
\phi, \mu , \eta$ such that $R \cdot R = 0$, and $C \cdot R = - L_{1} \,
Q(g,R)$ hold on $U_{S} \cap U_{C}$, provided that $(M,g)$ is a
non-quasi-Einstein manifold. We also mention that manifolds satisfying 
\begin{eqnarray}
C \cdot R &=& L \, Q(g,R)\,,  \label{R77}
\end{eqnarray}
were investigated in \cite{P72}. Furthermore, we have

\begin{theo}[\protect \cite{DHS}, Theorem 4.1 and Corollary 4.1]
Let $(M,g)$, $n \geq 4$, be a semi-Riemannian manifold. If $R \cdot C - C
\cdot R \, =\, L\, Q(g,C)$ holds on $U_{S} \cap U_{C} \subset M$, for some
function $L$, then $R \cdot R \, =\, L\, Q(g,R)$ and $C \cdot R \, =\, 0$ on
this set. In particular, if $R \cdot C \, =\, C \cdot R$ holds on $U_{S}
\cap U_{C}$ then $R \cdot R \, =\, R \cdot C \, =\, C \cdot R \, =\, 0$ on
this set.
\end{theo}

Quasi-Einstein hypersurfaces isometrically immersed in spaces of constant
curvature were investigated in \cite{DHS2, G6}, see also references therein.
In particular, in \cite{DHS2} an example of a quasi-Einstein hypersurface in
a semi-Riemannian space of constant curvature was found. More precisely, in
that paper it was shown that some warped product $\overline{M} \times _{F} 
\widetilde{N}$, with $\dim \overline{M} = 1$ and $\dim \widetilde{N}
\geqslant 4$, can be locally realized as a non-pseudosymmetric
Ricci-pseudosymmetric quasi-Einstein hypersurface in a semi-Riemannian space
of constant curvature. The difference tensor of that hypersurface is
expressed by a linear combination of the tensors $Q(g,R)$ and $Q(S,R)$.

\section{Warped product manifolds}

Warped products play an important role in Riemannian geometry (see e.g. \cite%
{K 2, ON}) as well as in the general relativity theory (see e.g. \cite{BE,
BEP, DK, ON}). Many well-known spacetimes of this theory, i.e. solutions of
the Einstein field equations, are warped products, e.g. the Schwarzschild,
Kottler, Reissner-Nordstr\"{o}m, Reissner-Nordstr\"{o}m-de Sitter, Vaidya,
as well as Robertson-Walker spacetimes. We recall that a warped product $%
\overline{M} \times _F \widetilde{N}$ of a $1$-dimensional manifold $(%
\overline{M}, \bar{g})$, $\bar{g}_{11} = - 1$, and a $3$-dimensional
Riemannian space of constant curvature $(\widetilde{N}, \tilde{g})$, with a
warping function $F$, is said to be a \textsl{Robertson-Walker spacetime}
(see e.g. \cite{BE, BEP, ON, Schmutzer}). It is well-known that the
Robertson-Walker spacetimes are conformally flat quasi-Einstein manifolds.
More generally, one also considers warped products $\overline{M} \times _F 
\widetilde{N}$ of $(\overline{M}, \bar{g})$, $\dim \, \overline{M} = 1$, $%
\bar{g}_{11} = - 1$, with a warping function $F$ and $(n-1)$-dimensional
Riemannian manifold $(\widetilde{N}, \tilde{g})$, $n \geqslant 4$. Such
warped products are called \textsl{generalized Robertson-Walker spacetimes} (%
\cite{ARS}, \cite{EJK}, \cite{San}). Curvature conditions of pseudosymmetry
type on such spacetimes have been considered among others in \cite{CDGP, DD,
DDHKS, 42, DeKuch, DeSche1, 33, HV02, Kow2}.

Let now $(\overline{M},\bar{g})$ and $(\widetilde{N},\tilde{g})$, $\dim \, 
\overline{M} = p$, $\dim \, \widetilde{N} = n-p$, $1 \leqslant p < n$, be
semi-Riemannian manifolds. Let $F: {\overline{M}} \rightarrow {\mathbb{R}}%
^{+}$ be a positive smooth function on $\overline{M}$. The warped product
manifold, in short warped product, $\overline{M} \times _F \widetilde{N}$ of 
$(\overline{M},\bar{g})$ and $(\widetilde{N}, \tilde{g})$ is the product
manifold $\overline{M} \times \widetilde{N}$ with the metric $g = \bar{g}
\times _F \tilde{g} $ defined by $\bar{g} \times _F \tilde{g} = {\pi}_1^{*} 
\bar{g} + (F \circ {\pi}_1)\, {\pi}_2^{*} \tilde{g}$, where ${\pi}_1 : 
\overline{M} \times \widetilde{N} \longrightarrow \overline{M}$ and ${\pi}_2
: \overline{M} \times \widetilde{N} \longrightarrow \widetilde{N}$ are the
natural projections on $\overline{M}$ and $\widetilde{N}$, respectively. In
this paper we consider warped products $\overline{M} \times _F \widetilde{N}$
with 1-dimensional base manifold $(\overline{M},\bar{g})$ and an $(n-1)$%
-dimensional fibre $(\widetilde{N}, \tilde{g})$, $n \geqslant 4$.

Let $\{ {\overline{U}} \times {\widetilde{V}} ; x^{1},x^2 = y^{1}, \ldots ,
x^{n} = y^{n-1} \} $ be a product chart for $\overline{M} \times \widetilde{N%
}$, where $\{ {\overline{U}};x^1 \} $ and $\{ {\widetilde{V}};y^{\alpha } \}$
are systems of charts on $(\overline{M},\bar{g})$ and $(\widetilde{N}, 
\tilde{g})$, respectively. The local components of the metric $g = \bar{g}
\times _F \tilde{g}$ with respect to this chart are the following: $g_{11} = 
\bar{g}_{11}=\varepsilon=\pm 1$, $g_{hk} = F\, \tilde{g}_{\alpha \beta }$ if 
$h = \alpha $ and $k = \beta $, and $g_{hk} = 0$ otherwise, $\alpha , \beta
, \gamma , \dots \in \{ 2, \ldots ,n \} $ and $h,i,j,k \ldots \in \{ 1,2,
\ldots ,n \} $. We will denote by bars (resp., by tildes) tensors formed
from $\bar{g}$ (resp., $\tilde{g}$). It is known that the local components $%
\Gamma ^{h} _{ij}$ of the Levi-Civita connection $\nabla $ of $\overline{M}
\times _F \widetilde{N}$ are the following: (see e.g. \cite{K 2, DeKuch}) 
\begin{eqnarray}  \label{VV1}
& &\Gamma ^{1} _{11}\ = 0\, , \  \  \  \  \Gamma ^{\alpha } _{\beta \gamma } \
=\  \widetilde{\Gamma } ^{\alpha } _{\beta \gamma }\, , \  \  \  \  \Gamma ^{1}
_{\alpha \beta } \ =\ - \frac{\varepsilon}{2}\, F^{\prime }\tilde{g}
_{\alpha \beta }\, , \\
& & \Gamma ^{\alpha } _{a \beta } \ =\  \frac{1}{2F} F^{\prime \alpha }
_{\beta }\, ,\  \  \  \  \Gamma ^{1} _{\alpha 1}\ =\  \Gamma ^{\alpha } _{11}\ =\
0\, ,\  \  \  \ F^{\prime }\ =\  \partial _1 F \ =\  \frac{\partial F}{\partial
x^{1}}\, .  \nonumber
\end{eqnarray}
The local components $R_{hijk}$ of the curvature tensor $R$ and the local
components $S_{hk}$ of the Ricci tensor $S$ of $\overline{M} \times _F 
\widetilde{N}$ which may not vanish identically are the following (see e.g. 
\cite{RRD 12, DeKuch}): 
\begin{equation}
R_{\alpha 11 \beta}=- \frac{1}{2}\, T_{11}\, \tilde{g}_{\alpha \beta} = - 
\frac{\mbox{{\rm tr}}\; T}{2}\, g_{11} \tilde{g}_{\alpha \beta}\,,\  \
R_{\alpha \beta \gamma \delta}=F ( \widetilde{R}_{\alpha \beta \gamma
\delta} - \frac{\Delta_1 F}{4 F}\, \widetilde{G}_{\alpha \beta \gamma
\delta})\,,  \label{W1}
\end{equation}
\begin{equation}
S_{11} = - \frac{n-1}{2\, F}\, T_{11}\,,\  \ S_{\alpha \beta }=\widetilde{S}%
_{\alpha \beta } - \bigl( \frac{\mbox{{\rm tr}}\; T}{2} + (n-2)\, \frac{
\Delta _1 F }{4\, F} \bigr)\, \tilde{g}_{\alpha \beta }\, ,  \label{W2}
\end{equation}
\begin{eqnarray}
& & T_{11}= F^{\prime \prime }- \frac{(F^{\prime 2}}{2\, F}\,, \  \ %
\mbox{{\rm tr}}\; T =\bar g^{11}T_{11} = \varepsilon \bigl(F^{\prime \prime
}- \frac{(F^{\prime 2}}{2\, F} \bigr)\,,  \nonumber \\
& & {\Delta}_1 F = {\Delta}_{1 \bar{g}} F = \bar{g}^{11}(F^{\prime 2 }=
\varepsilon (F^{\prime 2 }\,.  \label{W3}
\end{eqnarray}
The scalar curvature $\kappa $ of $\overline{M} \times _F \widetilde{N}$
satisfies the following relation 
\begin{equation}
\kappa = \frac 1F\Bigl( \tilde{\kappa} - (n - 1)\bigl( \mbox{{\rm tr}}\; T
+(n-2)\, \frac{ \Delta _1 F }{4\, F} \bigr) \Bigr) .  \label{W4}
\end{equation}

Using (\ref{qat}), (\ref{W1}) and (\ref{W2}) we can check that the local
components $Q(g,R)_{hijklm}$ and $Q(S,R)_{hijklm}$ of the tensors $Q(g,R)$
and $Q(S,R)$ which may not vanish identically are the following: 
\begin{eqnarray}  \label{qgr5}
Q(g,R)_{1\beta \gamma \delta 1 \mu } &=& F g_{11}\bigl(\widetilde
R_{\mu \beta \gamma \delta} +(\frac{ \mbox{{\rm tr}}\; T}{2}-\frac{\Delta_1F}{4F%
})\widetilde G_{\mu \beta \gamma \delta} \bigr)\,, \\
Q(g,R)_{\alpha \beta \gamma \delta \lambda \mu} &=& F^2\,Q(\tilde g,\widetilde
R)_{\alpha \beta \gamma \delta \lambda \mu}\,, \\
Q(S,R)_{1\beta \gamma \delta 1 \mu} &=& -\frac{\mbox{{\rm tr}}\; T }{2}
g_{11} \bigl((n-1) \widetilde R_{\mu \beta \gamma \delta} - \tilde
g_{\beta \gamma}\widetilde S_{\delta \mu} + \tilde g_{\beta \delta}\widetilde
S_{\gamma \mu} \\
&+& (\frac{\mbox{{\rm tr}}\; T }{2}-\frac{\Delta_1F}{4F})\widetilde
G_{\mu \beta \gamma \delta}\bigr)\,,  \nonumber \\
Q(S,R)_{1\beta 1\delta \lambda \mu} &=&-\frac{ \mbox{{\rm tr}}\; T }{2}
g_{11}\,Q(\tilde g,\widetilde S)_{\beta \delta \lambda \mu}\,, \\
Q(S,R)_{\alpha \beta \gamma \delta \lambda \mu} &=& F\,Q(\widetilde S,\widetilde
R)_{\alpha \beta \gamma \delta \lambda \mu} -\frac{\Delta_1F}{4}\,Q(\widetilde
S,\widetilde G)_{\alpha \beta \gamma \delta \lambda \mu} \\
&-&F\bigl( \frac{ \mbox{{\rm tr}}\; T }{2}+\frac{(n-2)\Delta_1F}{4F}\bigr)%
Q(\tilde g,\widetilde R)_{\alpha \beta \gamma \delta \lambda \mu}\,.  \nonumber
\end{eqnarray}

Let $V$ be the (0,4)-tensor with the local components $V_{hijk} =
g^{lm}S_{hl}R_{mijk}=S_h^{\ l}R_{lijk}$. Using (\ref{W1}) and (\ref{W2}) we
can verify that the only nonzero components of this tensor are the
following: 
\begin{eqnarray}  \label{V3}
V_{1\beta \gamma 1}&=&\frac{n-1}{4F}\, (\mbox{{\rm tr}}\; T )\, T_{11}\tilde
g_{\beta \gamma}=\frac{n-1}{4F}\, ({\mbox{{\rm tr}}\; T})^2 \, g_{11}\tilde
g_{\beta \gamma}\,, \\
V_{\alpha 11\delta}&=&-\frac{ \mbox{{\rm tr}}\; T }{2F}\,g_{11}\Bigl(%
\widetilde S_{\alpha \delta}-\bigl(\frac{ \mbox{{\rm tr}}\; T }{2}+\frac{%
(n-2)\Delta_1F}{4F}\bigr) \tilde g_{\alpha \delta}\Bigr)\,, \\
V_{\alpha \beta \gamma \delta}&=&\widetilde S_{\alpha}^{\  \epsilon}\widetilde
R_{\epsilon \beta \gamma \delta}- \bigl(\frac{ \mbox{{\rm tr}}\; T }{2}+\frac{%
(n-2)\Delta_1F}{4F}\bigr)\widetilde R_{\alpha \beta \gamma \delta}-\frac{%
\Delta_1F}{4F}(\tilde g_{\beta \gamma}\widetilde S_{\alpha \delta}-\tilde
g_{\beta \delta}\widetilde S_{\alpha \gamma}) \\
&+&\bigl(\frac{ \mbox{{\rm tr}}\; T }{2}+\frac{(n- 2)\Delta_1F}{4F}\bigr)%
\frac{\Delta_1F}{4F}\, \widetilde G_{\alpha \beta \gamma \delta}\,.  \nonumber
\end{eqnarray}
The last equality yields 
\begin{eqnarray}  \label{VRS}
V_{\alpha \beta \gamma \delta}+V_{\beta \alpha \gamma \delta}&=&\widetilde
S_{\alpha}^{\  \epsilon}\widetilde R_{\epsilon \beta \gamma \delta}+\widetilde
S_{\beta}^{\  \epsilon}\widetilde R_{\epsilon \alpha \gamma \delta}-\frac{%
\Delta_1F}{4F} (\tilde g_{\beta \gamma}\widetilde S_{\alpha \delta}-\tilde
g_{\beta \delta}\widetilde S_{\alpha \gamma}+ \tilde
g_{\alpha \gamma}\widetilde S_{\beta \delta}-\tilde g_{\alpha \delta}\widetilde
S_{\beta \gamma}) \\
&=&(\widetilde R\cdot \widetilde S)_{\alpha \beta \gamma \delta}- \frac{\Delta_1F%
}{4F}\,Q(\tilde g,\widetilde S)_{\alpha \beta \gamma \delta}\,.  \nonumber
\end{eqnarray}

Let $P$ be a (0,6)-tensor with local components 
\begin{eqnarray*}
P_{hijklm}&=&g_{hl}V_{mijk}-g_{hm}V_{lijk}- g_{il}V_{mhjk}+g_{im}V_{lhjk}
+g_{jl}V_{mkhi}-g_{jm}V_{lkhi} \\
&-&g_{kl}V_{mjhi}+g_{km}V_{ljhi}-
g_{ij}(V_{hklm}+V_{khlm})-g_{hk}(V_{ijlm}+V_{jilm}) \\
&+&g_{ik}(V_{hjlm}+V_{jhlm})+g_{hj}(V_{iklm}+V_{kilm})\,.
\end{eqnarray*}
The local components of $P$ which may not vanish identically are the
following: 
\begin{eqnarray}  \label{P2}
P_{1\beta 1\delta \lambda \mu}&=&g_{11}\bigl((\widetilde R\cdot \widetilde
S)_{\beta \delta \lambda \mu}+(\frac{ \mbox{{\rm tr}}\; T }{2} - \frac{\Delta_1F%
}{4F})Q(\tilde g,\widetilde S)_{\beta \delta \lambda \mu}\bigr)\,, \\
P_{1\beta \gamma \delta 1\mu}&=&g_{11}\Bigl(\widetilde S_{\mu}^{\
\epsilon}\widetilde R_{\epsilon \beta \gamma \delta}- \bigl(\frac{ \mbox{{\rm
tr}}\; T }{2}+\frac{(n-2)\Delta_1F}{4F}\bigr)\widetilde
R_{\mu \beta \gamma \delta} \\
&+&(\frac{ \mbox{{\rm tr}}\; T }{2}-\frac{\Delta_1F}{4F}) (\tilde
g_{\beta \gamma}\widetilde S_{\delta \mu}-\tilde g_{\beta \delta}\widetilde
S_{\gamma \mu})  \nonumber \\
&+&\bigl((n-2)(\frac{\Delta_1F}{4F})^2-\frac{ (\mbox{{\rm tr}}\; T)^{2}}{4} -%
\frac{(n-3) \mbox{{\rm tr}}\; T }{2}\, \frac{\Delta_1F}{4F}\bigr)\widetilde
G_{\mu \beta \gamma \delta}\Bigr)\,,  \nonumber
\end{eqnarray}
\begin{eqnarray}  \label{P3}
P_{\alpha \beta \gamma \delta \lambda \mu}&=&F\bigl(\tilde
g_{\alpha \lambda}V_{\mu \beta \gamma \delta}- \tilde
g_{\alpha \mu}V_{\lambda \beta \gamma \delta}-\tilde
g_{\beta \lambda}V_{\mu \alpha \gamma \delta} +\tilde
g_{\beta \mu}V_{\lambda \alpha \gamma \delta}+\tilde
g_{\gamma \lambda}V_{\mu \delta \alpha \beta} \\
&-&\tilde g_{\gamma \mu}V_{\lambda \delta \alpha \beta}-\tilde
g_{\delta \lambda}V_{\mu \gamma \alpha \beta} +\tilde
g_{\delta \mu}V_{\lambda \gamma \alpha \beta}-\tilde
g_{\beta \gamma}(V_{\alpha \delta \lambda \mu}+V_{\delta \alpha \lambda \mu}) 
\nonumber \\
&-&\tilde
g_{\alpha \delta}(V_{\beta \gamma \lambda \mu}+V_{\gamma \beta \lambda \mu})
+\tilde
g_{\beta \delta}(V_{\alpha \gamma \lambda \mu}+V_{\gamma \alpha \lambda \mu})
+\tilde g_{\alpha \gamma}(V_{\beta \delta \lambda \mu}+V_{\delta \beta \lambda \mu})%
\bigr)\,.  \nonumber
\end{eqnarray}

\section{Warped products with non-Einsteinian fibre}

Since we investigate non-Einstein and non-conformally flat manifolds
satisfying (\ref{a1}), we restrict our considerations to the set $\mathcal{U 
}= U \cap U_{S} \cap U_{C}$.

We assume that the warped product $\overline{M} \times _F \widetilde{N}$
satisfies (\ref{a1}) and the fiber $(\widetilde{N}, \tilde{g})$ is not
Einsteinian. Now we shall use the following identity which holds on any
semi-Riemannian manifold (\cite{DGHV}, Section 4) 
\begin{eqnarray}  \label{ge}
(n-2)(R \cdot C-C \cdot R)_{hijklm} &=& Q(S,R)_{hijklm}- \frac{\kappa}{n-1}%
\, Q(g,R)_{hijklm} + P_{hijklm} \,.
\end{eqnarray}
Thus, in view of (\ref{ge}) and the definition of the tensor $P$, condition (%
\ref{a1}) can be written in the form 
\begin{equation}
\bigl((n-2)L-1\bigr)\,Q(S,R)_{hijklm} = P_{hijklm} - \frac{\kappa}{n-1}%
\,Q(g,R)_{hijklm} \,.  \label{a2}
\end{equation}
For $h=1,\,i=\beta,\,j=1,\,k=\delta,\,l=\lambda,\,m=\mu$, in view of (\ref%
{qgr1})--(\ref{qgr5}), (\ref{a2}) yields 
\begin{equation}
\bigl((n-2)L-1\bigr)\,Q(S,R)_{1\beta 1\delta \lambda \mu}=P_{1\beta
1\delta \lambda \mu}\,.  \label{a3}
\end{equation}
Substituting (\ref{qgr4}) and (\ref{P1}) into (\ref{a3}) we obtain 
\begin{equation}
(\widetilde R\cdot \widetilde S)_{\beta \delta \lambda \mu}=\bigl(\frac{\Delta_1F%
}{4F}-\frac{(n- 2)L}{2}\, \mbox{{\rm tr}}\; T \bigr)\, Q(\tilde g,\widetilde
S)_{\beta \delta \lambda \mu}\,.  \label{RSL}
\end{equation}
On the other hand (\ref{a1}) implies 
\begin{eqnarray*}
& & (R \cdot C - C \cdot R)(X_{1},X_{2},X_{3},X_{4};X,Y) + (R \cdot C - C
\cdot R)(X,Y,X_{1},X_{2};X_{3},X_{4}) \\
&+& (R \cdot C - C \cdot R)(X_{3},X_{4},X,Y;X_{1},X_{2}) \ =\ 0\,,
\end{eqnarray*}
which in virtue of Proposition 4.1 of \cite{DGHV} is equivalent to 
\begin{eqnarray*}
& & (R \cdot C)(X_{1},X_{2},X_{3},X_{4};X,Y) + (R \cdot
C)(X,Y,X_{1},X_{2};X_{3},X_{4}) \\
&+& (R \cdot C)(X_{3},X_{4},X,Y;X_{1},X_{2}) \ =\ 0\,.
\end{eqnarray*}
Further, we have (\cite{CDGP}, section 3, eq. (3.19)) 
\begin{equation}
(\widetilde R\cdot \widetilde S)_{\beta \delta \lambda \mu}=(\frac{\Delta_1F}{4F}%
-\frac{ \mbox{{\rm tr}}\; T }{2})\, Q(\tilde g,\widetilde
S)_{\beta \delta \lambda \mu}\,.  \label{RS}
\end{equation}
Since $(\widetilde N,\tilde g),\  \dim \widetilde N\geqslant 3$, is not
Einsteinian, the tensor $Q(\tilde g,\widetilde S)$ is a non-zero tensor. Let 
$Q(\tilde g,\widetilde S)\neq 0$ at $x\in \mathcal{U}$. Thus on a coordinate
neighbourhood $V\subset \mathcal{U}$ of $x$, in virtue of (\ref{RSL}) and (%
\ref{RS}), we get 
\[
(L-\frac{1}{n-2})\, \mbox{{\rm tr}}\; T =0\,.
\]
We assert that $\mbox{{\rm tr}}\; T = 0$. Supposing that $\mbox{{\rm tr}}\;
T \neq 0$ at $y\in V$ we have $L=\frac{1}{n-2}$ on some neighbourhood $%
U_1\subset V$ of $y$. Therefore (\ref{a2}) reduces on $U_1$ to 
\begin{equation}
P=\frac{\kappa}{n-1}\,Q(g,R)\,.  \label{N3}
\end{equation}
Evidently, on $U_1$ we also have 
\begin{equation}
\frac{\Delta_1F}{4F}-\frac{ \mbox{{\rm tr}}\; T }{2}=const.  \label{C1}
\end{equation}
Now (\ref{N3}) gives 
\begin{equation}
P_{1\beta \gamma \delta 1\mu}=\frac{\kappa}{n- 1}\,Q(g,R)_{1\beta \gamma \delta
1\mu}\,.  \label{a4}
\end{equation}
Substituting into this equality (\ref{W4}), (\ref{qgr3}) and (\ref{P2}) we
obtain 
\begin{eqnarray}  \label{SR1}
\widetilde S_{\mu}^{\  \epsilon}\widetilde R_{\epsilon \beta \gamma \delta}&=&(%
\frac{\tilde \kappa}{n-1} -\frac{ \mbox{{\rm tr}}\; T }{2})\, \widetilde
R_{\mu \beta \gamma \delta}+(\frac{\Delta_1F}{4F}- \frac{ \mbox{{\rm tr}}\; T }{%
2})(\tilde g_{\beta \gamma}\widetilde S _{\mu \delta}-\tilde
g_{\beta \delta}\widetilde S_{\mu \gamma}) \\
&-&(\frac{\tilde \kappa}{n-1}-\frac{ \mbox{{\rm tr}}\; T }{2})(\frac{\Delta_1F%
}{4F} -\frac{ \mbox{{\rm tr}}\; T}{2} )\widetilde G_{\mu \beta \gamma \delta}\,.
\nonumber
\end{eqnarray}
Using now (\ref{C1}) and (\ref{SR1}) we see that $\mbox{{\rm tr}}\; T =
const.$ and consequently also $\frac{\Delta_1F}{4F}=const.$ on $U_1$.
Whence, after standard calculations, we deduce that $F$ must be of the form 
\begin{equation}
F(x^1)=(ax^1+b)^2\,,\  \ a,b\in \mathbb{R}\,.  \label{F}
\end{equation}
For such $F$ we have $\mbox{{\rm tr}}\; T = 0$ on $U_1$, a contradiction.
Therefore 
\begin{equation}
\mbox{{\rm tr}}\; T = 0  \label{N4}
\end{equation}
on $V$. Thus (\ref{C1}) reduces on $V$ to 
\begin{equation}
\frac{\Delta_1F}{4F}=const.=c_1\,.  \label{N5}
\end{equation}
Note that (\ref{N4}) and (\ref{N5}), in the same manner as above, imply (\ref%
{F}) and we have 
\begin{equation}
\mbox{{\rm tr}}\; T = 0\,,\  \  \frac{\Delta_1F}{4F}=c_1=\varepsilon a^2\,.
\label{T}
\end{equation}
We prove now that 
\begin{equation}
L=\frac{1}{n-2}  \label{L}
\end{equation}
on $V$. Applying (\ref{W4}), (\ref{qgr1}), (\ref{qgr2}), (\ref{P2}), (\ref%
{N4}) and (\ref{N5}) to 
\[
\bigl((n-2)L-1\bigr)\,Q(S,R)_{1\beta \gamma \delta 1\mu} =
P_{1\beta \gamma \delta 1\mu} -\frac{\kappa}{n-1}\,Q(g,R)_{1\beta \gamma \delta
1\mu}
\]
we get 
\begin{equation}  \label{SR2}
\widetilde S_{\mu}^{\  \epsilon}\widetilde R_{\epsilon \beta \gamma \delta}=%
\frac{\tilde \kappa}{n-1} \, \widetilde R_{\mu \beta \gamma \delta}+\varepsilon
a^2(\tilde g_{\beta \gamma}\widetilde S _{\mu \delta}-\tilde
g_{\beta \delta}\widetilde S_{\mu \gamma}) - \frac{\varepsilon \tilde \kappa a^2%
}{n-1}\, \widetilde G_{\mu \beta \gamma \delta}\,.
\end{equation}
But on the other hand (\ref{V3}), by (\ref{N4}) and (\ref{N5}), gives 
\[
V_{\alpha \beta \gamma \delta}=\widetilde S_{\alpha}^{\  \epsilon}\widetilde
R_{\epsilon \beta \gamma \delta} -(n-2)c_1\widetilde
R_{\alpha \beta \gamma \delta}-c_1(\tilde g_{\beta \gamma}\widetilde
S_{\alpha \delta} -\tilde g_{\beta \delta}\widetilde
S_{\alpha \gamma})+(n-2)c_1^2\widetilde G_{\alpha \beta \gamma \delta}\,,
\]
which by (\ref{SR2}) turns into 
\[
V_{\alpha \beta \gamma \delta}=\bigl(\frac{\tilde \kappa}{n-1}-(n-2)c_1\bigr)%
\widetilde R_{\alpha \beta \gamma \delta} +(n-2)c_1^2\widetilde
G_{\alpha \beta \gamma \delta}\,.
\]
Substituting this into (\ref{P3}) we obtain 
\begin{equation}
P_{\alpha \beta \gamma \delta \lambda \mu}=F\bigl(\frac{\tilde \kappa}{n-1}%
-(n-2)c_1\bigr)\,Q(\tilde g,\widetilde R)
_{\alpha \beta \gamma \delta \lambda \mu}\,.  \label{L2}
\end{equation}
Now 
\[
\bigl((n-2)L-1\bigr)Q(S,R)_{\alpha \beta \gamma \delta \lambda \mu}=P_{\alpha%
\beta \gamma \delta \lambda \mu} -\frac{\kappa}{n-1}\,Q(g,R)_{\alpha \beta \gamma%
\delta \lambda \mu}\,,
\]
by making use of (\ref{qgr2}), (\ref{L2}) and 
\[
\frac{\kappa}{n-1}=\frac 1F\, \bigl(\frac{\tilde \kappa}{n-1}-(n-2)c_1\bigr)\,,
\]
turns into 
\begin{equation}  \label{L3}
\bigl((n-2)L-1\bigr)Q(S,R)_{\alpha \beta \gamma \delta \lambda \mu}=0\,.
\end{equation}
Since $V\subset \mathcal{U}$ and in virtue of (\ref{qgr3}), (\ref{qgr4}) and (%
\ref{N4}) we have 
\[
Q(S,R)_{1\beta \gamma \delta 1\mu}=Q(S,R)_{1\beta 1\delta \lambda \mu}=0\,,
\]
at least one of the local components of $Q(S,R)_{\alpha \beta \gamma \delta%
\lambda \mu}$ must be non-zero. Therefore (\ref{L3}) implies (\ref{L}). Thus
we have proved

\begin{theo}
Let $\overline M\times_F\widetilde N$ be a warped product manifold with $1$%
-dimensional base manifold $(\overline M,\bar g)$ and non-Einstein $(n-1)$%
-dimensional fiber $(\widetilde N,\tilde g)$, $n \geqslant 4$. If (\ref{a1})
is satisfied on $\overline M\times_F\widetilde N$ then on the set $\mathcal{U%
}$ we have (\ref{SR2}) and 
\begin{equation}  \label{newL3}
L=\frac{1}{n-2}\,,\  \ F(x^1)=(ax^1+b)^2\,,\ a,b\in \mathbb{R}\,.
\end{equation}
\end{theo}

\begin{cor}
Let $\overline M\times_F\widetilde N$ be a generalized Robertson-Walker
spacetime with non-Einstein fiber $(\widetilde N,\tilde g)$, $n \geqslant 4$%
. If (\ref{a1}) is satisfied on $\overline M \times_F\widetilde N$ then (\ref%
{SR2}) and (\ref{newL3}) hold on $\mathcal{U}$.
\end{cor}

\begin{prop}
Under assumptions of Theorem 4.1 the fiber manifold $(\widetilde N,\tilde g)$
is a Ricci-pseudosymmetric manifold of constant type (see e.g. \cite{G6}),
precisely 
\begin{equation}
\widetilde R\cdot \widetilde S=\varepsilon a^2\,Q(\tilde g,\widetilde S)\,.
\label{D1}
\end{equation}
Moreover, if $n \geqslant 5$ then the difference tensor $\widetilde
R\cdot \widetilde C-\widetilde C\cdot \widetilde R$ of the fiber is expressed
by the Tachibana tensors $Q(\widetilde S,\widetilde R)$ and $Q(\tilde
g,\widetilde R)$, precisely we have 
\begin{equation}
(n-3) (\widetilde R\cdot \widetilde C-\widetilde C\cdot \widetilde R) =
Q(\widetilde S,\widetilde R) -\frac{\tilde \kappa}{(n-1)(n-2)}\,Q(\tilde
g,\widetilde R)\,.  \label{D3}
\end{equation}
\end{prop}

\begin{proof}
First we observe that (\ref{SR2}) implies (\ref{D1}). Applying the identity (%
\ref{ge}) to $(\widetilde N,\tilde g)$ we have 
\begin{equation}
(n-3)(\widetilde R\cdot \widetilde C-\widetilde C\cdot \widetilde
R)=Q(\widetilde S,\widetilde R)+\widetilde P- \frac{\tilde \kappa}{n-2}%
\,Q(\tilde g,\widetilde R)\,.  \label{D4}
\end{equation}
Using now (\ref{SR2}) we get 
\[
\widetilde V_{\alpha \beta \gamma \delta}+\widetilde
V_{\beta \alpha \gamma \delta}=\varepsilon a^2Q(\tilde g, \widetilde
S)_{\alpha \beta \gamma \delta}
\]
and 
\[
\widetilde P=\frac{\tilde \kappa}{n-1}\,Q(\tilde g,\widetilde R)-\varepsilon
a^2Q(\widetilde S,\widetilde G) -\varepsilon a^2\, \tilde g\wedge Q(\tilde
g,\widetilde S)\,,
\]
which by making use of $\tilde g\wedge Q(\tilde g,\widetilde
S)=-Q(\widetilde S,\widetilde G)$ (see (28) of \cite{DGHV}) reduces to $%
\widetilde P=\frac{\tilde \kappa}{n-1}\,Q(\tilde g,\widetilde R)$.
Substituting this equality into (\ref{D4}) we obtain (\ref{D3}).
\end{proof}

We have also the converse statement to Theorem 4.1.

\begin{theo}
Let $(\overline M,\bar g)$, $\bar g_{11}=\varepsilon$, be a $1$-dimensional
manifold and let $(\widetilde N,\tilde g)$ be an $(n-1)$-dimensional
non-Einstein manifold, $n \geqslant 4$, satisfying (\ref{SR2}). If $%
F(x^1)=(ax^1+b)^2$, then the warped product $\overline M\times_F\widetilde N$
fulfills (\ref{a1}) with $L=\frac{1}{n-2}$.
\end{theo}

\begin{proof}
As we have seen (cf. (\ref{a2})) (\ref{a1}) for $L=\frac{1}{n-2}$ takes the
form 
\begin{equation}
P=\frac{\kappa}{n-1}\,Q(g,R)\,.  \label{a5}
\end{equation}
Using now (\ref{VRS}), (\ref{T}) and (\ref{D1}) we have 
\begin{equation}
V_{\alpha \beta \gamma \delta}+V_{\beta \alpha \gamma \delta}=0\,.  \label{D2}
\end{equation}
Taking into account (\ref{V3}), (\ref{T}) and (\ref{SR2}) we obtain 
\[
V_{\alpha \beta \gamma \delta}=\phi \widetilde
R_{\alpha \beta \gamma \delta}+\psi \, \widetilde G_{\alpha \beta \gamma \delta}
\]
for some $\psi$ and $\phi=\frac{\tilde \kappa}{n-1}-\varepsilon(n- 2)a^2$.
Substituting the above equality and (\ref{D2}) into (\ref{P3}) we get 
\[
P_{\alpha \beta \gamma \delta \lambda \mu}=F\phi \,Q(\tilde g,\widetilde
R)_{\alpha \beta \gamma \delta \lambda \mu}\,,
\]
which in view of (\ref{W4}) and (\ref{T}) takes the form 
\[
P_{\alpha \beta \gamma \delta \lambda \mu}=\frac{\kappa}{n- 1}\,Q(g,R)_{\alpha%
\beta \gamma \delta \lambda \mu}\,.
\]
Using (\ref{P1}), (\ref{T}) and (\ref{D1}) we have $P_{1\beta
1\delta \lambda \mu}=0$, which means that (\ref{a3}) is satisfied. Finally, in
the same manner we obtain (\ref{a4}). Thus we see that (\ref{a5}) is
satisfied for all systems of indices.
\end{proof}

\begin{cor}
The equality (\ref{SR2}) is satisfied on every Einstein manifold $%
(\widetilde N,\tilde g)$. Thus every warped product $\overline
M\times_F\widetilde N$ with $1$-dimensional base $(\overline M,\bar g)$,
Einsteinian fiber $(\widetilde N,\tilde g)$, $\dim \widetilde{N} \geqslant 4$%
, which is not a space of constant curvature, and the warping function $%
F(x^1)=(ax^1+b)^2$ satisfies (\ref{a1}) with $L=\frac{1}{n-2}$.
\end{cor}

In the next section we also show that there exist warped products with
Einsteinian fiber, which is not a space of constant curvature, satisfying (%
\ref{a1}) with $L=\frac{1}{n-1}$.

\section{Warped products with Einsteinian fibre}

In this section we consider warped products $\overline M\times_F\widetilde N$%
, $\dim \overline{M} = 1$, assuming that a fibre $(\widetilde N,\tilde g)$
is an Einstein manifold, i.e. 
\begin{equation}
\widetilde S_{\alpha \beta}=\frac{\tilde \kappa}{n-1}\, \tilde
g_{\alpha \beta}\,.  \label{E}
\end{equation}
Using (\ref{W2}) we can easily show that such warped product is a
quasi-Einstein manifold. It is worth to noticing that $\widetilde R\neq 
\frac{\tilde \kappa}{(n-1)(n-2)}\widetilde G$ on $\mathcal{U}$. Using (\ref%
{qgr4}), (\ref{P1}) and (\ref{E}), we get 
\begin{equation}
Q(S,R)_{1\beta 1 \delta \lambda \mu}= P_{1\beta 1\delta \lambda \mu}=0\,.
\label{B1}
\end{equation}
Analogously, in view of (\ref{qgr3}), (\ref{P2}) and (\ref{E}), we have 
\begin{equation}  \label{B2}
Q(S,R)_{1\alpha \beta \gamma 1\delta}=\frac{ \mbox{{\rm tr}}\; T }{2}\,g_{11}\,%
\Bigl(-(n-1)\widetilde R_{\delta \alpha \beta \gamma}+ \bigl(\frac{\tilde \kappa%
}{n-1}-(\frac{ \mbox{{\rm tr}}\; T }{2}-\frac{\Delta_1F}{4F})\bigr)%
\widetilde R_{\delta \alpha \beta \gamma}\Bigr)\,,
\end{equation}
\begin{eqnarray}  \label{B3}
& & P_{1\beta \gamma \delta 1\mu} \ =\ g_{11}\, \Bigl(\eta \widetilde
R_{\mu \beta \gamma \delta} \\
&+&\bigl((\frac{ \mbox{{\rm tr}}\; T }{2}-\frac{\Delta_1F}{4F})\frac{%
\tilde \kappa}{n-1} +(n-2)(\frac{\Delta_1F}{4F})^2-\frac{ (\mbox{{\rm tr}}\;
T)^{2}}{4} -(n-3)\frac{ \mbox{{\rm tr}}\; T }{2}\, \frac{\Delta_1F}{4F}\bigr)%
\widetilde G_{\mu \beta \gamma \delta}\Bigr)\,  \nonumber
\end{eqnarray}
where 
\[
\eta = \frac{\tilde \kappa}{n-1} - \frac{ \mbox{{\rm tr}}\; T }{2} - \frac{%
(n-2)\Delta_1F}{4F} \,.
\]
Finally, making use of (\ref{V3}) and (\ref{E}), we obtain 
\[
V_{\alpha \beta \gamma \delta}=\eta \,(\widetilde R_{\alpha \beta \gamma \delta} -%
\frac{\Delta _1F}{4F}\, \widetilde G_{\alpha \beta \gamma \delta})
\]
and next, in virtue of (\ref{qgr5}) and (\ref{P3}), also 
\begin{equation}
Q(S,R)_{\alpha \beta \gamma \delta \lambda \mu}=F\eta \,Q(\tilde g,\widetilde
R)_{\alpha \beta \gamma \delta \lambda \mu}\,,  \label{B4}
\end{equation}
\begin{equation}
P_{\alpha \beta \gamma \delta \lambda \mu}=F\eta \,Q(\tilde g,\widetilde
R)_{\alpha \beta \gamma \delta \lambda \mu}\,.  \label{B5}
\end{equation}
Thus taking into account (\ref{qgr1}), (\ref{qgr2}) and (\ref{B1})--(\ref{B5}%
), we see that (\ref{a2}) is equivalent to the following two equalities: 
\begin{eqnarray}
& & \bigl((n-2)L-1\bigr)\, \frac{ \mbox{{\rm tr}}\; T }{2}\, \Bigl(%
-(n-1)\widetilde R +\bigl(\frac{\tilde \kappa}{n-1}-(\frac{ \mbox{{\rm tr}}\;
T }{2}-\frac{\Delta_1F}{4F})\bigr)\, \widetilde G\Bigr)  \nonumber \\
&=& \frac{ \mbox{{\rm tr}}\; T }{2}\, \bigl(\widetilde R+(\frac{ \mbox{{\rm
tr}}\; T}{2}-\frac{\Delta_1F}{4F})\, \widetilde G\bigr)\,,  \label{B6}
\end{eqnarray}
\begin{equation}
\bigl((n-2)L-1\bigr)\, \eta \,Q(\tilde g,\widetilde R)=\frac{ \mbox{{\rm tr}}%
\; T }{2}\,Q(\tilde g,\widetilde R)\,.  \label{B7}
\end{equation}
We consider two cases: (i) $\mbox{{\rm tr}}\; T = 0$ and (ii) $\mbox{{\rm
tr}}\; T \neq 0$.

(i) $\mbox{{\rm tr}}\; T = 0$. Since $Q(\tilde g,\widetilde R)\neq 0$ on $%
\mathcal{U}$, so (\ref{B7}) leads to 
\[
\bigl((n-2)L-1\bigr)\bigl(\frac{\tilde \kappa}{n-1}-\frac{(n-2)\Delta_1F}{4F}%
\bigr)=0\,.
\]
Supposing that $\frac{\tilde \kappa}{n-1}=\frac{(n-2)\Delta_1F}{4F}$ and
using (\ref{B1}), (\ref{B2}) and (\ref{B4}) we see that $Q(S,R)=0$, a
contradiction. Thus we get $L=\frac{1}{n-2}$. Moreover, solving the
differential equation $\mbox{{\rm tr}}\; T = 0$, one can see that the
warping function $F$ must be of the form (\ref{F}). Thus we have the
situation described in Corollary 4.1.

(ii) $\mbox{{\rm tr}}\; T \neq 0$. Now (\ref{B6}) leads to 
\[
(n-2)\bigl((n-1)L-1\bigr)\, \widetilde R=\Bigl((n-2)L\bigl(\frac{\tilde \kappa%
}{n-1} -(\frac{ \mbox{{\rm tr}}\; T}{2}-\frac{\Delta_1F}{4F})\bigr)-\frac{%
\tilde \kappa}{n-1}\Bigr)\widetilde G\,.
\]
Whence $L=\frac{1}{n-1}$ and 
\begin{equation}
\frac{\Delta_1F}{4F}-\frac{ \mbox{{\rm tr}}\; T }{2}\ =\  \frac{\tilde \kappa}{%
(n-1)(n-2)}\,.  \label{B8}
\end{equation}
It is worth to noticing that under above equalities (\ref{B7}) also holds. (%
\ref{B8}), in view of (\ref{W3}), takes the form 
\begin{equation}
F\,F^{\prime \prime }-(F^{\prime 2}+2\varepsilon C_1F\ =\ 0\,,\  \ C_1\ =\ 
\frac{\tilde \kappa}{(n-1)(n-2)}\,.  \label{B9}
\end{equation}
This is exactly equation (29) of \cite{DeSche1}. We can check that the
following functions are solutions of (\ref{B9}) (cf. \cite{DeSche1}, Lemma
3.1): 
\begin{eqnarray}
F(x^{1}) &=& \varepsilon C_{1}\, ( x^{1} + \frac{\varepsilon c}{C_{1}}%
)^{2}\, , \  \  \  \  \varepsilon C_{1} > 0 \, ,  \nonumber \\
F(x^{1}) &=& \frac{c}{2}\bigl( \exp(\pm \frac{b}{2} x^{1}) - \frac{2
\varepsilon C_{1}}{b^{2} c} \exp( \mp \frac{b}{2} x^{1}) \bigr)^{2}\, ,\  \  \
\ c > 0\, ,\  \  \  \ b \neq 0\, ,  \label{B10} \\
F(x^{1}) &=& \frac{2 \varepsilon C_{1}}{c}\, \bigl( 1 + \sin (c x^{1} + b)%
\bigr)\, , \  \  \  \  \frac{ \varepsilon C_{1}}{c} > 0\, ,  \label{B11}
\end{eqnarray}
where $b$ and $c$ are constants and $x^{1}$ belongs to a suitable non-empty
open interval of ${\mathbb{R}}$. The first form of $F$ leads to $\mbox{{\rm
tr}}\; T = 0$ and must be excluded in our case. Thus we have proved

\begin{theo}
Let $\overline M\times_F\widetilde N$ be a warped product manifold with $1$%
-dimensional base manifold $(\overline M,\bar g)$ and Einsteinian $(n-1)$%
-dimensional fiber $(\widetilde N,\tilde g)$. If (\ref{a1}) is satisfied on $%
\overline M\times_F\widetilde N$, then on the set $\mathcal{U}$ we have:
either 
\[
L=\frac{1}{n-2}\  \ and\  \ F(x^1)=(ax^1+b)^2\,,\ a,b\in \mathbb{R}\,,
\]
either 
\begin{eqnarray*}
L=\frac{1}{n-1}\  \ and\  \ F(x^{1}) &=& \frac{c}{2}\bigl( \exp(\pm \frac{b}{2}
x^{1}) - \frac{2 \varepsilon C_{1}}{b^{2} c} \exp( \mp \frac{b}{2} x^{1}) %
\bigr)^{2}\, ,\  \ c > 0\,,\  \ b \neq 0\, , \\
or\  \ F(x^{1}) &=& \frac{2 \varepsilon C_{1}}{c}\, \bigl( 1 + \sin (c x^{1} +
b)\bigr)\, , \  \  \frac{ \varepsilon C_{1}}{c} > 0\,.
\end{eqnarray*}
\end{theo}

Taking into account the proof of Theorem 5.1 we easily obtain

\begin{cor}
If $F$ is of the form (\ref{B10}) or (\ref{B11}) then $\overline{M} \times
_{F} \widetilde{N}$ with $1$-dimensional base manifold and Einstein fibre
satisfies (\ref{a1}) with $L = \frac{1}{n-1}$.
\end{cor}

\section{Examples}

Corollaries 4.1 and 5.1 give rise to examples of warped products satisfying (%
\ref{a1}) with Einstein fibre. The problem of finding of a warped product
satisfying (\ref{a1}) with non-Einstein fibre reduces, via Theorem 4.2, to
the problem of finding of an example of a semi-Riemannian manifold $(%
\widetilde{N},\widetilde{g})$, $\dim \widetilde{N} = n-1 \geqslant 3$,
fulfilling (\ref{SR2}). To obtain a suitable example we will use results of 
\cite{DeGlo1, DGHV,DH102}. First of all, we adopt to our consideration
results contained in Theorem 3.1 of \cite{DeGlo1} and in Theorem 3.2 of \cite%
{DGHV}. Those results we can present in the following

\begin{theo}
Let $(\widetilde{N},\widetilde{g})$ be a hypersurface isometrically immersed
in a semi-Riemannian space of constant curvature $N_{s}^{n}(c)$, $n
\geqslant 4$, with signature $(s,n-s)$, where $c = \frac{\tau }{(n-1)n}$, $%
\tau$ is the scalar curvature of the ambient space and $\widetilde{g}$ is
the metric tensor induced on $\widetilde{N}$. Moreover, let the second
fundamental tensor $H$ of $\widetilde{N}$ satisfies on some non-empty
connected set $\widetilde{U} \subset \widetilde{N}$ the equation 
\begin{equation}
H^{3}\ =\ tr(H)\, H^{2} + \lambda \, H,  \label{E1}
\end{equation}
where $\lambda $ is some function on $\widetilde{U}$ and let the constant $%
\varepsilon = \pm 1$ be defined by the Gauss equation of $\widetilde{N}$ in $%
N_{s}^{n}(c)$, i.e. by 
\begin{eqnarray*}
\widetilde{R} \ =\  \frac{\varepsilon}{2}\, H \wedge H +\frac{\tau}{(n-1)n}\, 
\widetilde{G}\,.
\end{eqnarray*}
(i) (cf. \cite{DeGlo1}, Theorem 3.1) On $\widetilde{U}$ we have 
\begin{equation}
\widetilde{S}_{\mu}^{\  \epsilon}\widetilde R_{\epsilon \beta \gamma \delta} \
=\  \mu \, \bigl( \widetilde{R}_{\mu \beta \gamma \delta} - \frac{\tau }{(n-1)n}
\, \widetilde{G}_{\mu \beta \gamma \delta}\bigr) + \frac{\tau }{(n-1)n} \, (%
\widetilde{g}_{\beta \gamma}\widetilde{S}_{\mu \delta} - \widetilde{g}%
_{\beta \delta}\widetilde{S}_{\mu \gamma}),  \label{E2}
\end{equation}
where $\mu = \frac{(n-2) \tau }{(n-1)n} - \varepsilon \, \lambda$. \newline
(ii) (cf. \cite{DGHV}, Theorem 3.2) If $n \geqslant 5$ then on $\widetilde{U}
$ we have 
\begin{equation}
(n-3)\, (\widetilde{R} \cdot \widetilde{C} - \widetilde{C} \cdot \widetilde{R%
}) \ =\ Q(\widetilde{S}, \widetilde{R}) + \bigl(\frac{(n-2) \tau }{(n-1) n}
- \varepsilon \, \lambda - \frac{\widetilde{\kappa}}{n-2}\bigr) Q(\widetilde{%
g}, \widetilde{R})\,,  \label{E3}
\end{equation}
where $\widetilde{\kappa}$ is the scalar curvature of $\widetilde{N}$.
\end{theo}

We note that (\ref{E2}) implies immediately that $\widetilde{R} \cdot 
\widetilde{S} = \frac{\tau }{(n-1)n} \, Q(\tilde{g}, \widetilde{S})$. In
addition, if we assume that on $\widetilde{U}$ we have 
\begin{equation}
\lambda = 0\  \  \  \  \mbox{and}\  \  \  \ (n-2) \tau = n \widetilde{\kappa}
\label{E4}
\end{equation}
then (\ref{SR2}) holds on $\widetilde{U}$. The last remark suggests a
solution of our problem. Namely, the last two conditions are realized on the
hypersurface presented in Example 5.1 of \cite{DH102}. Let $(M,g)$ be the
manifold defined in Example 5.1 of \cite{DH102}. We denote it by $(%
\widetilde{N}, \widetilde{g})$. Clearly $(\widetilde{N}, \widetilde{g})$ is
a manifold of dimension $\geqslant 4$. However, it is easy to verify that if
we repeat the construction of $(\widetilde{N}, \widetilde{g})$ for the $3$%
-dimensional case then all curvature properties remain true excluding, of
course, properties expressed by its Weyl conformal curvature tensor. Thus
without loss of generality we can assume that $\dim \, \widetilde{N} = n-1
\geqslant 3$. In Example 5.1 of \cite{DH102}, among other things, it was
shown that $(\widetilde{N}, \widetilde{g})$ is locally isometric to a
hypersurface in a semi-Riemannian space of non-zero constant curvature.
Since our considerations are local, we can assume that $(\widetilde{N}, 
\widetilde{g})$ is a hypersurface isometrically immersed in that space.
Since $(\widetilde{N}, \widetilde{g})$ fulfils (\ref{E4}), Theorem 5.1
finishes our construction. We note that by making use of (\ref{E3}) and ({%
\ref{E4}), we obtain (\ref{D3}). \newline
}

\noindent \textbf{Remark 6.1.} (i) The Roter type warped products $\overline
M\times_F\widetilde N$ with $1$-dimensional base manifold $(\overline
M,\overline g)$ and non-Einstein $(n-1)$-dimensional fiber $(\widetilde
N,\widetilde g)$, $n \geqslant 4$, were investigated in \cite{DeSche1}.
Among other results it was proved that the curvature tensor $\widetilde{R}$
of the fiber $(\widetilde N,\widetilde g)$ is expressed by the
Kulkarni-Nomizu tensors $\widetilde{S} \wedge \widetilde{S}$, $\widetilde{g}
\wedge \widetilde{S}$ and $\widetilde{G}$, i.e. the fiber also is a Roter
type manifold, provided that $n \geqslant 5$. Therefore, if we assume that
the fiber manifold $(\widetilde{N}, \widetilde{g})$ considered in Theorem
6.1 is a non-pseudosymmetric Ricci-pseudosymmetric hypersurface, for
instance the Cartan hypersurfaces of dimension $6$, $12$ or $24$ have this
property (see e.g. \cite{G6}), then fibers of both constructions are
non-isometric. \newline
(ii) From (\ref{R77}), by a suitable contraction, we get 
\begin{eqnarray}
C \cdot S &=& L \, Q(g,S)\,.  \label{R777}
\end{eqnarray}
We refer to \cite{Kow1} and \cite{P72} for examples of warped products
satisfying (\ref{R777}). The condition (\ref{R777}) holds on some
hypersurfaces in semi-Riemannian space forms, and in particular, on the
Cartan hypersurfaces (\cite{DeGlo1}, Theorems 3.1 and 4.3). Recently,
hypersurfaces in semi-Euclidean space satisfying (\ref{R777}) were
investigated in \cite{Ozg}. \newline
(iii) We also can investigate semi-Riemannian manifolds $(M,g)$, $n
\geqslant 4$, satisfying on $U_{C} \subset M$ the following condition of
pseudosymmetric type (see e.g. \cite{DDP, DH01}) 
\begin{eqnarray}
R \cdot R - Q(S,R) &=& L\, Q(g,C)\, ,  \label{R877}
\end{eqnarray}
where $L$ is some function on this set. Warped products satisfying (\ref%
{R877}) were investigated in \cite{DDP}. Among other results, in \cite{DDP}
it was shown that this condition is satisfied on every $4$-dimensional
warped product $\overline M\times_F\widetilde N$ with $1$-dimensional base.
Thus in particular, every $4$-dimensional generalized Robertson-Walker
spacetime satisfies (\ref{R877}). We mention that (\ref{R877}) holds on
every hypersurface in a semi-Riemannian space of constant curvature (see
e.g. \cite{DGHSaw}, eq. (4.4)). \newline
(iv) In \cite{DGHS} (Example 4.1) a warped product $\overline
M\times_F\widetilde N$ of an $(n-1)$-dimensional base $(\overline
M,\overline g)$, $n \geqslant 4$, and an $1$-dimensional fibre $(\widetilde
N,\widetilde g)$ satisfying $\mbox{rank}\, S = 1$, $\kappa = 0$, $R \cdot R
= 0$ and $C \cdot S = 0$ was constructed. In addition, we can easily check,
that (\ref{a1}) with $L = \frac{1}{n-2}$ and $Q(S,C) = Q(S,R)$ hold on $%
\overline M \times_F \widetilde N$ (\cite{DGHZ01}). Therefore on $\overline
M \times_F \widetilde N$ we also have $(n-2)\, (R\cdot C - C \cdot R) =
Q(S,C)$. Semi-Riemannian manifolds satisfying $R\cdot C - C \cdot R = L \,
Q(S,C)$, for some function $L$, are investigated in \cite{DGHZ01}. An
example of a quasi-Einstein non-Ricci-simple manifold satisfying the last
condition is given in Section 6 of \cite{DGHSaw}.

\section{Conclusions}

Let $\overline M\times_F\widetilde N$ be the warped product of an $1$%
-dimensional manifold $(\overline M,\overline g)$, $\overline g_{11} =
\varepsilon = \pm 1$, the warping function $F: {\overline{M}} \rightarrow {%
\mathbb{R}}^{+}$ and an $(n-1)$-dimensional, $n \geqslant 4$,
semi-Riemannian manifold $(\widetilde N, \widetilde g)$.

If $(\widetilde N, \widetilde g)$ is a semi-Riemannian space of constant
curvature then $\overline M\times_F\widetilde N$ is a quasi-Eintein
conformally flat pseudosymmetric manifold. Evidently, the Friedmann-Lema%
\^{\i}tre-Robertson-Walker spacetimes belong to this class of manifolds.
Further, if the fibre $(\widetilde N, \widetilde g)$, $n \geqslant 5$, is an
Einstein manifold, which is not of constant curvature, then $\overline
M\times_F\widetilde N$ is a quasi-Einstein non-conformally flat
non-pseudosymmetric Ricci-pseudosymmetric manifold. In this case the
difference tensor $R\cdot C - C \cdot R$ is expressed by a linear
combination of the Tachibana tensors $Q(g,R)$ and $Q(S,R)$ (\cite{CDGP}).

If the fibre $(\widetilde N, \widetilde g)$, $n \geqslant 4$, is a
conformally flat Ricci simple manifold such that its scalar curvature $%
\widetilde{\kappa}$ vanishes then $\overline M\times_F\widetilde N$ is a
non-conformally flat pseudosymmetric manifold, provided that $F = F(x^{1}) =
\exp x^{1}$ (\cite{DDHKS}, Proposition 4.2 and Example 4.1). In addition we
have $(n-1)\, (R\cdot C - C \cdot R) = Q(S,C)$ (\cite{DGHZ01}).

If the fibre $(\widetilde N, \widetilde g)$, $n \geqslant 4$, is some Roter
type manifold and the warping function $F$ satisfies (\ref{B9}) then $%
\overline M\times_F\widetilde N$ is a Roter type manifold, and in a
consequence a non-conformally flat pseudosymmetric manifold (\cite{DeSche1},
Theorem 5.1). As it was mentioned in Section 2, the tensor $R\cdot C - C
\cdot R$ is expressed by a linear combination of some Tachibana tensors.

The above presented facts show that under some conditions imposed on the
fibre or the fibre and the warping function of a generalized
Robertson-Walker spacetime such spacetime is a pseudosymmetric or
Ricci-pseudosymmetric manifold and its difference tensor $R\cdot C - C \cdot
R$ is expressed by a linear combination of some Tachibana tensors. In this
paper we consider an inverse problem. Namely, if the tensors $R\cdot C - C
\cdot R$ and $Q(S,R)$ are linearly dependend on a generalized
Robertson-Walker spacetime then we determine the warping function, as well
as curvature properties of the fibre of such spacetime. In the case when the
considered generalized Robertson-Walker spacetimes are $4$-dimensional
manifolds, it is possible to apply the algebraic classification of
space-times satisfying some conditions of pseudosymmetry type given in \cite%
{P109}, see also \cite{EriSen, HV02}.

\vspace{5mm} {\footnotesize \noindent }

{\footnotesize 
\begin{tabular}{llll}
Kadri Arslan, Ridvan Ezenta\d s & Ryszard Deszcz & Marian Hotlo\'{s} &  \\ 
and Cengizhan Murathan &  &  &  \\ 
Department of Mathematics & Department of Mathematics & Institute of
Mathematics &  \\ 
Art and Science Faculty & Wroc\l aw University & and Computer Science &  \\ 
Uluda\v{g} University & of Environmental & Wroc\l aw University &  \\ 
16059 Bursa, TURKEY & and Life Sciences & of Technology &  \\ 
\textsf{arslan@uludag.edu.tr } & Grunwaldzka 53 & Wybrze\.{z}e Wyspia\'{n}%
skiego 27 &  \\ 
\textsf{rezentas@uludag.edu.tr} & 50-357 Wroc\l aw, POLAND & 50-370 Wroc\l %
aw, POLAND &  \\ 
\textsf{cengiz@uludag.edu.tr} & \textsf{ryszard.deszcz@up.wroc.pl} & \textsf{%
marian.hotlos@pwr.wroc.pl} & 
\end{tabular}
}

\end{document}